\documentclass[preprint,12pt]{elsarticle}
\usepackage{amsfonts,amsmath,amssymb}
\usepackage{MnSymbol}
\usepackage{graphicx}
\usepackage{amsthm}
\usepackage{euscript}
\usepackage{color}

\newtheorem{theorem}{Theorem}
\newtheorem{lemma}{Lemma}

\newtheorem{definition}{Definition}

\newcommand{\Mu}{\zeta}
\newcommand{\tT}{\intercal}
\newcommand{\bigo}[1]{O\left( #1 \right) }

\def\l{\left}
\def\r{\right}

\def\RR{\mathbb{R}}

\begin{document}
\begin{frontmatter}

\title{Numerical Solutions to Singular Reaction-Diffusion Equation over Elliptical Domains}
\author{Matthew A. Beauregard\footnote{Department of Mathematics, Clarkson University, Potsdam, NY 13699, USA({\tt mbeaureg@clarkson.edu})}}

\begin{abstract}
Solid fuel ignition models, for which the dynamics of the temperature is independent of the single-species mass fraction, attempt to follow the dynamics of an explosive event.  Such models may take the form of singular, degenerate, reaction-diffusion equations of the quenching type, that is, the temporal derivative blows up in finite time while the solution remains bounded.  Theoretical and numerical investigations have proved difficult for even the simplest of geometries and mathematical degeneracies.  Quenching domains are known to exist for piecewise smooth boundaries, but often lack theoretical estimates.  Rectangular geometries have been primarily studied.  Here, this acquired knowledge is utilized to determine new theoretical estimates to quenching domains for arbitrary piecewise, smooth, connected geometries. Elliptical domains are of primary interest and a Peaceman-Rachford splitting algorithm is then developed that employs temporal adaptation and nonuniform grids.  Rigorous numerical analysis ensures numerical solution monotonicity, positivity, and linear stability of the proposed algorithm.  Simulation experiments are provided to illustrate the accomplishments.
\end{abstract}



\end{frontmatter}

\section{Introduction}

Nonlinear evolution equations that form a singularity in finite time are ubiquitous in nature. Applications are broad appearing in models from chemistry, physics, biology, rocket, and combustion engineering.  Solid fuel ignition models attempt to approximate an explosive event, identified as the rate of change of the temperature increasing without bound and forming a singularity in finite time.  The formation of the singularity equates to ignition in the combustor.  Moreover, the finite time ignition is triggered if the temperature reaches a critical threshold. The sophistication of combustion models can be reduced considerably if one examines asymptotically close to the ignition time.  In such cases, the equation for the temperature decouples from the chemistry and mass-species fraction \cite{Bebernes_1989,Poinsot_2005}.  The final result is a highly nonlinear differential equation, in particular,
\begin{eqnarray}\label{Introduction:GeneralEq1}
s(x,y) u_t &=& \Delta u + f(u), ~ t>0,~ (x,y,t)\in \Omega \times (0,T), \\
\label{Introduction:GeneralEq2}u(x,y,0)&=&u_0(x,y),~ (x,y)\in \Omega,
\end{eqnarray}
where $\Omega\in\mathbb{R}^2,$ Dirichlet boundary conditions specified on the boundary $\partial \Omega,$ and $0\leq u_0(x,y)<1$ for $(x,y)\in\Omega.$ The source term, $f(u)$, is a highly nonlinear function. It is positive and strictly increasing for $0\leq u < 1.$  Most importantly, it tends to infinity when $u \rightarrow 1^-.$  The degeneracy function $s(x,y)\geq 0$ and models particular heat transportation characteristics within the domain.  While it remains positive within the domain it may vanish or become singular on a countable subset of points on the boundary. If the degeneracy function vanishes this is seen as a defect in the transportation of heat, while if it becomes singular it is a defect in the diffusion of heat. In either case, the degeneracy indicates a defect at a particular location of the combustor or, common to combustion engine designs, a special material concentration point \cite{Beauregard_2011,Kirk_2002,Poinsot_2005}.

A sufficient condition for the singularity to form in the temporal derivative is that $u \rightarrow 1^-$ as $t\rightarrow T^-<\infty$ in $\tilde{\Omega}\subseteq\Omega$ \cite{Sheng_2012}. In such cases, the solution $u$ is said to quench at the \textit{quenching time} $T$ over the domain $\Omega.$ The set $\tilde{\Omega}$ is called the {\em quenching location set,\/} and in this study involves a single point in the domain. On one hand, given an initial condition and domain shape there is no guarantee that the solution will quench.  One the other hand, it is known that for a piecewise smooth and connected domain there does exists a critical size for which the solution will always quench for all sizes larger \cite{Chan_2011}.  This was originally observed and proven in the one-dimensional case by Kawarada \cite{Kawarada_1975} for a particular $f(u),$ and then was extended to general source terms by Levine and Montgomery \cite{Levine_1980}. In higher dimensions, the results are not nearly as precise and finding estimates to the critical size that quenching persists for a particular shape is problematic. Fortunately, there are known theoretical estimates for rectangular domains \cite{Chan_1994}.  Numerical approximations have been employed to extend and resolve these estimates further \cite{Beauregard_2011,Liang_2006}. In this paper, stemming from the creation of lower and upper bound solutions to the singular problem, theoretical estimates of the critical quenching domain are established in \S 2 for arbitrary domains.  This provides an avenue of creating novel estimates for arbitrary domains, in particular when the domain shape is convex and will be used to provide additional verification of the numerical algorithm in the experiments of \S 6.

This paper is also interested in the development and analysis of an adaptive splitting algorithm that can accurately be used to explore the singular problem posed over an elliptical domain.  To date, numerical explorations have primarily focused on rectangular domains while circular and elliptical geometries may be more applicable to realistic applications.  In addition, the analysis of \S 2 provides valuable estimates to the critical quenching domains and are used to further validate the numerical scheme in the experiments of \S 6. Moreover, the intricate numerical methodologies, such as splitting, adaptation, and nonuniform grids requires extensions to such geometric considerations. Indeed, after careful design, the numerical analysis shows that the numerical scheme is indeed reliable, accurate, monotonically increases toward a steady state or quenching, and is weakly stable.

The paper is organized as follows. In the following section theoretical estimates for quenching domains for piecewise smooth, connected domains are given. It is evident that these results can be extended to blow-up problems.  In addition, as a corollary to the theorems, quenching time estimates can be established for arbitrary domain shapes and sizes. In \S 3 the study begins its focus on elliptical domains.  The equations are then written in elliptical coordinates.  Hence, the computation is designed over the rectangular grid generated by the coordinate transformation.  Appropriate boundary conditions are then established.  In \S 4 the adaptive, second order, splitting algorithm is given.  The numerical analysis of the algorithm is in \S 5.  In \S 6 numerous numerical experiments are offered, namely a statistical analysis of introduced errors, critical quenching domain calculations for elliptical domains of certain ratios of minor to major axes, and experimentally studying the effect of the mathematical degeneracy on the solution's fundamental characteristics. The paper is then concluded in \S 7.

\section{Theoretical Estimates of Quenching Domains}

For singular, reaction-diffusion equations of the quenching type one of most interesting features is the existence of a critical quenching domain size for a particular domain shape.  As one might expect, showing that this is indeed the case for arbitrary domains is a difficult task, especially in higher dimensions.  Recently, in \cite{Chan_2011}, it has been shown that for connected piecewise smooth domains there exists a critical quenching domain size. In realistic domain shapes, this novel result indicates this phenomena will persist, however, theoretical estimates of the critical size are lacking.  Even so, reliable theoretical estimates have been established for rectangular domains.  These estimates have been verified numerically in a number of resources \cite{Beauregard_2011,Sheng_2003,Liang_2006}.

The use of upper and lower solutions to differential equations has a long history in analysis of blow-up problems \cite{Bandle_1998}.  The fundamental idea can be traced back to \cite{Sattinger_1973}.  Consider the solution $u(x,y,t)$ of (\ref{Introduction:GeneralEq1})-(\ref{Introduction:GeneralEq2}).

Let $w(x,y,t)$ be a time-dependent lower bound of $u(x,y,t).$  Then if $w(x,y,t)\rightarrow 1$ in finite time $T$ then the quenching singularity forms in $u(x,y,t)$ at a quenching time of $T^*\leq T.$ Therefore, if one can establish a time-dependent lower bound of $u$ that approaches a value of one in finite time then $u$ will quench. Similarly, let $v(x,y,t)$ be a time-dependent upper bound of $u(x,y,t).$  Then if $v(x,y,t)<1$ for all time then clearly $u(x,y,t)$ will not quench.

\begin{definition}
A \textit{lower solution} of (\ref{Introduction:GeneralEq1})-(\ref{Introduction:GeneralEq2}) is a function $v(x,y,t)$ that satisfies
\begin{eqnarray*}
v_t &\leq& \Delta v + f(v),~~~ (x,y)\in \Omega,~t>0,\\
v&\leq& 0, ~ (x,y)\in \partial \Omega,~ t>0,\\
v&\leq& 0, ~ (x,y)\in \bar{\Omega},~ t=0.
\end{eqnarray*}
\end{definition}

\begin{definition}
An \textit{upper solution} of (\ref{Introduction:GeneralEq1})-(\ref{Introduction:GeneralEq2}) is a function $w(x,y,t)$ that satisfies
\begin{eqnarray*}
w_t&\geq& \Delta w + f(w), ~~~ (x,y)\in \Omega,~ t>0,\\
w&\geq& 0, ~ (x,y)\in \partial \Omega, ~ t>0,\\
w&\geq& 0, ~ (x,y)\in \bar{\Omega},~ t=0,
\end{eqnarray*}
\end{definition}

In this light, upper and lower solutions provide a valid tool of analysis for such problems. This idea is utilized in creating estimates to the critical size for arbitrary, connected, piecewise smooth domain shapes. Consider Nagumo's lemma (\cite{Walter_1970}, \S 24),

\begin{lemma}
Let $u(x,y,t)$ satisfy (\ref{Introduction:GeneralEq1})-(\ref{Introduction:GeneralEq2}) with reactive term $f(u)$.
If $v(x,y,t)$ is a lower solution then $v\leq u$ for $(x,y)\in \bar{\Omega}$ and $t>0.$ Similarly, if
$w(x,y,t)$ is an upper solution then $w\geq u$ for $(x,y)\in\bar{\Omega}$ and $t>0.$
\end{lemma}

Clearly, as stated earlier, if the lower solution approaches one in the domain then $u$ will quench.  Given the shape of the domain $\Omega$ there exists a critical size for which the lower solution approaches a value of one.  Thus the lower solution's critical quenching domain size is a lower bound to the critical quenching domain size for $u.$  On the other hand, if the upper solution does not quench, that is approach a value of one, then neither does $u.$  Hence, the upper solution's critical quenching domain size is an upper bound to the critical quenching domain for $u.$

Define $\Omega_{rect}^{(low)}$ as the maximal rectangle that is a subset of $\Omega.$ Similarly, define $\Omega_{rect}^{(upper)}$ as the minimal rectangle that $\Omega$ is a subset of. Two theorems are now established.
\begin{theorem}
Let
\begin{eqnarray*}
v(x,y,t) = \left\{ \begin{array}{cl}{v^{(1)}(x,y,t)}&{(x,y)\in \Omega_{rect}^{(low)}, ~ t>0}\\{0}&{~~~~~\bar{\Omega} \setminus \Omega_{rect}^{(low)}, ~ t>0}\end{array}\right. ,
\end{eqnarray*}
where $v^{(1)}$ satisfies
\begin{eqnarray*}
v_t^{(1)} &=& \Delta v^{(1)} + f(v^{(1)}),~~~ (x,y)\in \Omega_{rect}^{(low)}~, ~t>0,\\
v^{(1)}&=& 0, ~~~~~~~~~~~~~~~~~~~~ (x,y)\in \partial \Omega_{rect}^{(low)},~ t>0,\\
v^{(1)}&=& 0, ~~~~~~~~~~~~~~~~~~~~ (x,y)\in \bar{\Omega}_{rect}^{(low)}~,~ t=0.
\end{eqnarray*}
then $v$ is a lower solution.
\end{theorem}
\begin{proof}
Since $f(0)>0$ then
\begin{eqnarray*}
v_t =\left\{ \begin{array}{ll}{v_t^{(1)}=\Delta v^{(1)} + f(v^{(1)})}&{(x,y)\in \Omega_{rect}^{(low)}, ~ t>0}\\{0\leq \Delta v + f(v)}&{~~~~~~\Omega\setminus\Omega_{rect}^{(low)}}\end{array}\right. .
\end{eqnarray*}
By definition, $v=0$ on $\partial \Omega$ and is zero initially.
\end{proof}

\begin{theorem}
Let $w(x,y,t)$ be a solution to
\begin{eqnarray*}
w_t &=& \Delta w + f(w),~~~ (x,y)\in \Omega_{rect}^{(upper)}~, ~t>0,\\
w&=& 0, ~~~~~~~~~~~~~~~ (x,y)\in \partial \Omega_{rect}^{(upper)},~ t>0,\\
w&=& 0, ~~~~~~~~~~~~~~~ (x,y)\in \bar{\Omega}_{rect}^{(upper)}~,~ t=0.
\end{eqnarray*}
Then $w$ restricted to the domain $\bar{\Omega}_{rect}^{(upper)} \bigcap~ \bar{\Omega}$ is an upper solution to $u.$
\end{theorem}
\begin{proof}
Since $f(w)>0$ and is monotonically increasing then it is assured that $w\geq 0$ on $\Omega_{rect}^{(upper)}.$ Thus $w\geq 0$ on $\partial \Omega \subseteq \bar{\Omega}_{rect}^{(upper)}.$
\end{proof}

The above theorems imply that the quenching time can be bounded by that of quenching times found for rectangular domains.  This offers reliable bounds for convex domains, hence providing much needed direction and insight to verify numerical calculations of quenching domains.  It seems plausible that sufficient improvements can also be made for non-convex domains. The bounds provide a valid check to the consistency of quenching computations.  While the application here are motivated by quenching problems, the above results are easily transferrable to other singular problems, in particular blow-up problems.

These theorems will be used to help verify the numerical experiments of \S 6.  To illustrate the theorems, consider an elliptical geometry, that is,
\[ \Omega = \{ (x,y) | \frac{x^2}{A^2}+\frac{y^2}{B^2}<1\},\]
for some positive constants $A$ and $B$.  Then the maximal inscribed rectangle is
\[ \Omega_{rect}^{low} = \{ (x,y) | -\frac{A}{\sqrt{2}} < x < \frac{A}{\sqrt{2}} \& -\frac{B}{\sqrt{2}} < y < \frac{B}{\sqrt{2}}\}.\]
Similarly, the minimum rectangle that can be inscribed in $\Omega$ is
\[ \Omega_{rect}^{upper} = \{ (x,y) | -A < x < A \& 0 -B < y < B.\]
Let $f(u)=1/(1-u).$ For these rectangular shapes the critical quenching domain has been established \cite{Beauregard_2011,Sheng_2003}.  Hence, Theorem 2.2 can be used to offer theoretical bounds for particular ratios of minor and major axes $B/A$.  The results are offered in Table  \ref{TableQuenchingBoundsTheory}.

\begin{table}[h]
\begin{center}
\begin{tabular}{l|r|c|c}\hline
$B/A$ & Rect. Area & Elliptical Area Bounds             & AB Bounds \\ \hline
.125  & 18.8054    & 14.7697 $\leq \pi AB \leq$ 29.5395 & 4.7013 $\leq AB \leq$ 9.4027\\ \hline
.250  &  9.6722    &  7.5965 $\leq \pi AB \leq$ 15.1931 & 2.4181 $\leq AB \leq$ 4.8361\\ \hline
.375  &  6.8501    &  5.3801 $\leq \pi AB \leq$ 10.7601 & 1.7125 $\leq AB \leq$ 3.4251\\ \hline
.500  &  5.5986    &  4.3971 $\leq \pi AB \leq$  8.7943 & 1.3997 $\leq AB \leq$ 2.7993\\ \hline
.625  &  4.9679    &  3.9018 $\leq \pi AB \leq$  7.8036 & 1.2420 $\leq AB \leq$ 2.4840\\ \hline
.750  &  4.6453    &  3.6484 $\leq \pi AB \leq$  7.2968 & 1.1613 $\leq AB \leq$ 2.3226\\ \hline
.875  &  4.4964    &  3.5315 $\leq \pi AB \leq$  7.0629 & 1.1241 $\leq AB \leq$ 2.2482\\ \hline
\end{tabular}
\caption{Estimated quenching domains for ellipses with minor and major axis of $B$ and $A,$ respectively. For particular ratios $B/A$, the third column bounds the critical quenching area of the ellipse.  The last column shows the theoretical bounds for the critical product $AB$, which is related to third column by a multiple of $pi$.  As a comparison, the second column shows what the known critical quenching area of a rectangle with length to width ratio of $B/A.$ For brevity, quenching time bounds are not included.}
\label{TableQuenchingBoundsTheory}
\end{center}
\end{table}

\section{Elliptical Geometries}

Let $\Omega = \{ (x,y) ~|~ x^2/A^2 + y^2/B^2 < 1 \}.$ Without loss of generality, assume that $A$ and $B$ are the major and minor axes of the ellipse, respectively.  Hence $A>B.$  The cartesian and physical domain is mapped to a rectangular domain through the elliptical coordinate transformation, that is,
\begin{eqnarray}\label{GovEqs:EllipDef}
&& x=a \cosh(\mu) \cos(\theta), ~~~ y=a \sinh(\mu) \sin(\theta),
\end{eqnarray}
where $\theta\in[0,2\pi),$ $\mu\in [0,\Mu].$ The focal distance $a$ and $\Mu$ are uniquely determined through $A$ and $B$,
\begin{eqnarray*}
A=a\cosh(\Mu),~ B=a\sinh(\Mu).
\end{eqnarray*}
Let $S=\{ (\mu,\theta) ~|~ 0<\mu<\Mu,~ 0<\theta<2\pi\}$ and $\partial S$ be its boundary. Equations (\ref{Introduction:GeneralEq1})-(\ref{Introduction:GeneralEq2}) are written in elliptical coordinates,
\begin{eqnarray}\label{GovEqs:Ellip1}
~~~~~~~~~~~s(\mu, \theta)u_t &=& \phi(\mu,\theta)\l( u_{\mu\mu} + u_{\theta\theta}\r) + f(u),~ (\mu,\theta,t)\in S \times (0,T), \\
u(\mu,\theta,0)&=& u_0(\mu,\theta),~~ (\mu,\theta)\in S, \\
\label{GovEqs:Ellip3}u(\Mu,\theta,t)&=& 0,~~~ \theta\in(0,2\pi),~ t\in (0,T),
\end{eqnarray}
where,
\begin{eqnarray*}
\phi(\mu,\theta)=\frac{1}{a^2\l(\sinh^2(\mu) + \sin^2(\theta)\r)}
\end{eqnarray*}
is the Jacobian of the transformation. Periodic boundary conditions are assumed in $\theta.$

The result of the transformation requires the determination of a boundary condition for $\mu=0.$  Recall (\ref{GovEqs:EllipDef}).  Notice that if $(\mu,\theta)\mapsto (-\mu, 2\pi - \theta)$ the cartesian point $(x,y)$ is unchanged. Hence, $u(\mu,\theta)=u(-\mu, 2\pi-\theta).$  This useful fact has already been exploited in numerical approximations utilizing fourier series for elliptic partial differential equations \cite{Lia_2004}.

\section{Adaptive Numerical Splitting Scheme}

In higher dimensions there is a tremendous increase in the computational cost.  However, this increase can be greatly attenuated using splitting techniques, for which offer efficient, reliable, and effective ways to compute the numerical solution.  The idea is well-established and publicized, that is, convert the costly high-dimensional problem into a set of lower-dimensional ones that can be solved accordingly.  The splitting procedure may not rely on the commutativity of the involved matrix operators and the resulting error is, ideally, at the order of the original discretization.  While splitting procedures is an established approach, coupling splitting with temporal and/or spatial adaptation is not a common feature.  Hence, a primary motivator of this research is to further develop adaptive methods that incorporate operator splitting.

Let $v(t)$ denote the numerical solution of (\ref{GovEqs:Ellip1})-(\ref{GovEqs:Ellip3}) through a semidiscretization
in space. Then
\begin{eqnarray}\label{GovEqs:MOL}
v'(t) &=& C v(t) + g(v(t)),~ t\in(0,T),
\end{eqnarray}
where $v(0)$ is the initial condition stemming from $u_0(\mu,\theta),~ C=P+R,$ and $P$ and $R$ are matrices resulting from the discretization in the $\mu$ and $\theta$ directions.  The nonlinearity $g(v(t)$ is the discretized reactive term $f(u).$ The degeneracy and the Jacobian are also contained in these matrices.  An implicit solution is found in terms of the matrix exponential,
\begin{eqnarray*}
v(t) = \exp( C t) v(0) + \int_0^t \exp( C (t - s) g(v(s)) ds.
\end{eqnarray*}
A second order, Peaceman-Rachford splitting technique is combined with a trapezoidal rule for the integration to approximate the solution,
\begin{eqnarray}\label{GovEqs:PRsplit}
v(t) = S(t)v(0) + \frac{t}{2} \l( g(v(t)) + S(t)g(v(0))\r),
\end{eqnarray}
where
\begin{eqnarray}\label{GovEqs:PRsplittingoperator}
S(t) = (I - \frac{t}{2} R)^{-1}(I - \frac{t}{2} P)^{-1}(I + \frac{t}{2} P)(I + \frac{t}{2} R).
\end{eqnarray}
The Peaceman-Rachford splitting technique has been well documented in the literature and in depth study of its stability properties can be found in \cite{Schatzman_1999}. However, incorporating adaptation with splitting is still not fully understood and, again, is a primary motivator of this  research.

\subsection{Fully Discretized Scheme}

Nonuniform centered differences in the $\mu$ and $\theta$ directions are utilized.  Let $\theta_j=\theta_{j-1} + \Delta \theta_{j-1}$ for $j=1,\ldots,M+1,$ where $\theta_0=0$ and $\theta_{M+1}=2\pi.$  Similarly, let $\mu_i = \mu_{i-1} + \Delta \mu_{i-1}$ for $i=1,\ldots, N+1,$ where $\Delta \mu_i$ is the nonuniform discretization size, $\mu_{N+1}=\Mu,$ and $\mu_0 = -\Delta \mu_0/2.$ Defining $\mu_0$ and $\mu_1$ in this manner allows for straightforward use of the symmetry condition $u(\mu_0,\theta_j)=u(-\mu_1, 2\pi-\theta_j)$ for the boundary condition at the $\mu=0$ wall.  Furthermore, it avoids an auxiliary condition, as that found in utilizing polar coordinates \cite{Beauregard_2013CIRC}, at $\mu=0$ and $\theta=n \pi$ for $n=0$ and $1.$  Further, let $u_{i,j}(t)$ be an approximation of the exact solution of (\ref{GovEqs:Ellip1})-(\ref{GovEqs:Ellip3}) at the grid point $(\mu_i,\theta_j,t).$  Hence the semidiscretized equations to the solid fuel ignition model are,
\begin{eqnarray*}
\dot{u}_{i,j} &=& \psi_{i,j} \l( \delta_{\mu}^2 + \delta_{\theta}^2\r) u_{i,j}+\frac{1}{s_{i,j}}f(u_{i,j}),
\end{eqnarray*}
for $i=1,\ldots,N,$ $j=0,\ldots,M$, $\psi_{i,j}=\phi_{i,j}/s_{i,j},$ and $\delta_{\mu}^2$ and $\delta_{\theta}^2$ are nonuniform centered difference operators, namely,
\begin{eqnarray*}
\delta_{\mu}^2 u_{i,j} &=& \frac{2}{\Delta \mu_{i-1}(\Delta \mu_i + \Delta \mu_{i-1})}u_{i-1,j} - \frac{2}{\Delta \mu_i \Delta \mu_{i-1}} u_{i,j} + \frac{2}{\Delta \mu_i (\Delta \mu_i + \Delta \mu_{i-1})}u_{i+1,j},~~~~ \\
\delta_{\theta}^2 u_{i,j} &=& \frac{2}{\Delta \theta_{j-1}(\Delta \theta_j + \Delta \theta_{j-1})}u_{i,j-1} - \frac{2}{\Delta \theta_j \Delta \theta_{j-1}} u_{i,j} + \frac{2}{\Delta \theta_j (\Delta \theta_j + \Delta \theta_{j-1})}u_{i,j+1}.
\end{eqnarray*}
The $N(M+1)$ equations are augmented by the boundary conditions,
\begin{eqnarray*}
u_{N+1,j}=0, ~ u_{0,j} = u_{1,M+1-j}, ~ u_{i,-1}=u_{i,M}, ~ u_{i,M+1}=u_{i,0}.
\end{eqnarray*}
Denote $v(t)=(u_{1,0}, u_{2,0}, \ldots, u_{N,0}, u_{1,1}, \ldots, u_{N,M})^{\tT}.$  The semidiscretized equations mirror (\ref{GovEqs:MOL}) with $g(v)=Q f(v),$ where $Q$ is a $N(M+1)\times N(M+1)$ diagonal matrix, that is, $ Q = \mbox{diag}\l(Q^{(j)}\r),~ Q^{(j)} = \mbox{diag}\l(s_{i,j}^{-1}\r).$ The matrices $P$ and $R$ are both $N(M+1)\times N(M+1)$ matrices.  Define $T_{\mu}$ and $T_{\theta}$ as the $N\times N$ and $(M+1)\times (M+1)$ tridiagonal matrices stemming from the $\mu$ and $\theta$ difference operators, respectively, with two-sided Dirichlet boundary conditions.  Define $I_N$ and $I_{M+1}$ as the $N\times N$ and $(M+1)\times (M+1)$ identity matrices.  Further, define
\begin{eqnarray*}
C_{\mu} &=& \left(\begin{array}{ccccc}
{1}&{0}&{\ldots}&{0}&{0}\\
{0}&{0}&{\ldots}&{0}&{1}\\
{\vdots}&{\ddots}&{\ddots}&{\udots}&{0}\\
{0}&{0}&{1}&{0}&{0}\\
{0}&{1}&{0}&{\ldots}&{0}\end{array}\right)\in \mathbb{R}^{(M+1)\times (M+1)},\\
C_{\theta} &=& \left(\begin{array}{ccccc}
{0}&{0}&{\ldots}&{0}&{\kappa_1}\\
{\vdots}&{\ddots}&{\ddots}&{0}&{0}\\
{\vdots}&{\ddots}&{\ddots}&{\vdots}&{\vdots}\\
{0}&{0}&{\ldots}&{0}&{0}\\
{\kappa_2}&{0}&{\ldots}&{0}&{0}\end{array}\right)\in \mathbb{R}^{(M+1)\times (M+1)},\\
\hat{I}_N &=& \left(\begin{array}{cccc}
{\kappa_3}&{0}&{\ldots}&{0}\\
{0}&{0}&{\ddots}&{\vdots}\\
{\vdots}&{\ddots}&{\ddots}&{\vdots}\\
{0}&{0}&{\ldots}&{0}\end{array}\right)\in \mathbb{R}^{N\times N},
\end{eqnarray*}
where
\[ \kappa_1=\frac{2}{\Delta \theta_M(\Delta \theta_0 + \Delta \theta_M)}, ~~~~\kappa_2 = \frac{2}{\Delta \theta_M(\Delta \theta_{M-1} + \Delta \theta_M)}, ~~~ \kappa_3=\frac{2}{\Delta \mu_0(\Delta \mu_1 + \Delta \mu_0)}.\]
The matrices $P$ and $R$ are compactly written,
\begin{eqnarray*}
P &=& B\l( I_{M+1}\otimes T_{\mu} + C_{\mu} \otimes \hat{I}_N\r), \\
\label{Rmat}R &=& B\l( T_{\theta}\otimes I_{N} + C_{\theta} \otimes I_N\r),
\end{eqnarray*}
where $B=\mbox{diag}\l(B^{(j)}\r),~ B^{(j)} = \mbox{diag}\l(\psi_{i,j}\r).$  Thus, in light of (\ref{GovEqs:PRsplit}) and (\ref{GovEqs:PRsplittingoperator}) the fully discretized equations are,
\begin{eqnarray}\label{NumericalScheme:FullyDiscretized}
v_{k+1} = S(\tau_k))v_k + \frac{\tau_k}{2} \l( g(w_{k+1}) + S(\tau_k)g(v_k)\r),
\end{eqnarray}
where $\tau_k=t_{k+1}-t_k$ is the nonuniform temporal step and $w_{k+1}$ is a suitable approximation to $v_{k+1}$.  Commonly, $w_{k+1}$ is calculated through an explicit Euler approximation, however this choice reduces the overall order of the approximation to unity.  In particular, a midpoint method can be used.  From Taylor expansions, let
$$ w_{k+1} = g(v_k) + \tau_k J_k~ q(v_k),$$
where
\begin{eqnarray*}\label{NumericalScheme:FullyDiscretizedMidPoint2}
q(v_k) &=& \l[ I + \frac{\tau_k}{2}\l( C + J_k \r) \r] \l( C v_k + g(v_k) \r),
\end{eqnarray*}
and $J_k$ is the Jacobian of $g(v_k),$ a diagonal matrix. Thus a second order, one-step method can be established.
\begin{eqnarray}\label{NumericalScheme:FullyDiscretizedMidPoint}
v_{k+1} = S(\tau_k))\l( v_k + \frac{\tau_k}{2} g(v_{k})\r) + \frac{\tau_k}{2} \l(~ g(v_k) + \tau_k J_k~ q(v_k)~\r)
\end{eqnarray}

\subsection{Adaptation}

As the solution increases the rate of change function may become unbounded, that is quench.  In such situations it becomes necessary to reduce the temporal step in order to better capture the quenching time and solution dynamics.  The manner for which the reduction in the temporal step should be done is debateable.  However, it seems plausible that if the temporal derivative is monitored then equidistributing its arc-length can provide a reasonable reduction.  This approach was pioneered for quenching problems in \cite{Sheng_2003} and experimental results have indicated favorable results.

The idea of equidistribution is not a new concept and was first introduced by \cite{Boor_1973}.  In the current situation, the equidistribution approach amounts to establishing the temporal step such that no
characteristic line grows faster than a user-specified tolerance. In particular, the temporal
step may be selected such that the resulting change in the solution is not greater than the
change experienced along that characteristic in the previous advancement. Hence, each
characteristic generates a possible temporal step for which the new temporal step is taken
as the minimal value over that set. This results in an implicit algebraic equation that is used to construct the new temporal step, that is,
$$\l(v_{k+1}'-v_{k}'\r)^2\cdot {e}_j + \l(\tau_k^{(j)}\r)^2 =
\l( v_{k}'-v_{k-1}' \r)^2\cdot {e}_j + \tau_{k-1}^2,$$
where $v_{\ell},~v_{\ell}'$ are the solution and derivative vectors defined through
(\ref{NumericalScheme:FullyDiscretizedMidPoint2}) at temporal levels $\ell,$ respectively, and $e_j\in\RR^N$ is the
$j$th unit vector, $1\leq j\leq N.$ The notation $\l({v}\r)^p$ means that each of the vector's
components is raised to the power $p.$ The initial step $\tau_0$ is given. The updated
temporal step $\tau_k$ is taken to be the minimum, that is,
$$\tau_k^2 = \tau_{k-1}^2 + \min_j\l\{\l[\l( v_k'-v_{k-1}' \r)^2-
\l( v_{k+1}'-v_k' \r)^2\r]\cdot e_j\r\},~~~ k=1,2,\ldots.$$
Of course, the above procedure requires a sufficient amount of data to be accumulated. This presents no difficulty as the temporal adaptation happens late in the computation when the necessary information is available.  If this is not the case, then a simple reduction in the temporal step can be performed to generate the required information. Needless to say, a healthy amount of recent publications on semi-adaptive methods, including
\cite{Beauregard_2011,Beauregard_2012MIT,Sheng_2003,Liang_2006,Qiao_2011}, implement such a procedure.

\section{Numerical Analysis}
A fundamental feature of the continuous solution of (\ref{Introduction:GeneralEq1})-(\ref{Introduction:GeneralEq2}) is that, provided $s(x,y)\geq 0$, it monotonically increases toward a steady state or quenches in finite time \cite{Chan_2011,Nouaili_2011}.  Further, it remains positive provided the initial condition is initially greater than or equal to zero.  This fundamental feature of the mathematical model is a requirement of any reliable numerical computation. Ultimately, this rests on proving necessary conditions that guarantee positivity of the involved matrices.  While it is recognized that alternative approaches to guarantee a positive and monotonically increasing numerical solution is possible, this approach is sharp in the sense that a violation of any of the hypothesis results in a loss of invertibility of particular matrices, thus rendering the calculation useless. Here, a new and general theorem is presented that solidifies the necessary requirements to guarantee positivity and monotonicity of the numerical solution for any discretization that takes the form of (\ref{NumericalScheme:FullyDiscretized}).

\begin{theorem}\label{GeneralMonotonicThm}
For any beginning time step $\ell\geq 0$ and initial condition $v_{\ell}>0$ such that $C v_{\ell} + g(v_{\ell})>0$, if the matrices
\[ (I - \frac{\tau_k}{2} R)^{-1}, ~(I - \frac{\tau_k}{2} P)^{-1}\]
are positive,
\[(I + \frac{\tau_k}{2} P), ~ (I + \frac{\tau_k}{2} R),~ (I+\tau_k C)\]
are nonnegative, $g(v)$ is a strictly increasing and positive function for $v\in[0,1)$, and $\tau_k$ is sufficiently small, then the sequence $\{v_k\}_{k\geq \ell}$ generated by (\ref{NumericalScheme:FullyDiscretizedMidPoint}) monotonically increases until unity is exceeded by a component of the solution vector $v_k$ or steady state is attained for both constant and variable $\tau_k$, $k\geq \ell$.
\end{theorem}
\begin{proof}
Consider a uniform temporal step $\tau_{\ell}.$ From (\ref{NumericalScheme:FullyDiscretized}) we have,
\begin{eqnarray*}
v_{\ell+1}-v_{\ell} &=& S(\tau_{\ell})\l(v_{\ell} + \frac{\tau_{\ell}}{2} g(v_{\ell})\r)+\frac{\tau_{\ell}}{2} g(w_{\ell+1})-v_{\ell}, \\
&=& \tau_{\ell} (I - \frac{\tau_{\ell}}{2} R)^{-1}(I - \frac{\tau_{\ell}}{2} P)^{-1} p_{\ell},
\end{eqnarray*}
where
\begin{eqnarray*}
p_{\ell}&=& C v_{\ell} + \frac{1}{2}\l( (I + \frac{\tau_{\ell}}{2} P)(I + \frac{\tau_{\ell}}{2} R)g(v_{\ell}) + (I -\frac{\tau_{\ell}}{2} P)(I - \frac{\tau_{\ell}}{2} R)g(w_{\ell+1})\r) \\
&=& C v_{\ell} + g(v_{\ell}) + \tau_{\ell}\l( \frac{1}{4} g(v_{\ell}) + \frac{1}{2}(I - \frac{\tau_{\ell}}{2} P)(I - \frac{\tau_{\ell}}{2} R)J_{\ell} q(v_{\ell})\r).
\end{eqnarray*}
Given $C v_{\ell} + g(v_{\ell}) > 0$ and provided that $\tau_{\ell}$ is sufficiently small
then $p_{\ell}>0.$  Hence, as the matrices $(I - \frac{\tau_{\ell}}{2} R)$ and $(I - \frac{\tau_{\ell}}{2} P)$ are inverse positive then $v_{\ell+1} - v_{\ell} >0.$  Therefore the sequence increases in the first temporal step.

Next, let $v_{k+1}$ for $k>\ell$ be the numerical solution stemming from a uniform $\tau_k\equiv \tau_{\ell}.$  Assume $v_k - v_{k-1} > 0.$ Consider the difference
\begin{eqnarray*}
v_{k+1}-v_{k} &=& S(\tau_{\ell})\l( v_k - v_{k-1} + \frac{\tau_{\ell}}{2}(g(v_k)-g(v_{k-1}))\r) + \frac{\tau_{\ell}}{2}(g(w_{k+1})-g(w_{k}))
\end{eqnarray*}
Now,
\begin{eqnarray*}
w_{k+1}-w_k &=& (v_k - v_{k-1}) + \tau_{\ell} (q_k - q_{k-1}) \\
&=& (I+\tau_{\ell}C)(v_k - v_{k-1}) + \tau_{\ell} (g_k - g_{k-1}) + \bigo{\tau_{\ell}^2}
\end{eqnarray*}
is greater than zero as $g$ strictly increases and the matrix $(I+\tau_{\ell}C)$ is assumed positive.  The second order terms can also be shown to be positive, however, it is not considered here as the error in the method is second order. Hence $v_{k+1}-v_{k}>0.$ This completes the proof for a uniform temporal step by mathematical induction.

Finally, consider a nonuniform $\tau_k$ for $k>\ell.$  Assume $v_k - v_{k-1}>0$ and examine the difference
\begin{eqnarray*}
v_{k+1}-v_{k} &=& \tau_{k} (I - \frac{\tau_{k}}{2} R)^{-1}(I - \frac{\tau_k}{2} P)^{-1} (C v_k + g(v_k)) + \bigo{\tau_k^2}.
\end{eqnarray*}
It is necessary to show that $C v_k + g(v_k)>0,$ which is accomplished in the following manner,
\begin{eqnarray*}
C v_k + g(v_k) &=& C v_{k-1} + g(v_{k-1}) + g(v_k)-g(v_{k-1}) + C(v_k - v_{k-1})\\
&=& C v_{k-1} + g(v_{k-1}) + g(v_k)-g(v_{k-1}) \\
& & + C\l( \tau_{k-1}(I - \frac{\tau_{k-1}}{2} R)^{-1}(I - \frac{\tau_k-1}{2} P)^{-1}\right.\\
& & \left.\times  (C v_{k-1} + g(v_{k-1})) + \bigo{\tau_{k-1}^2}\r)\\
&=& S(\tau_{k-1}) (C v_{k-1} + g(v_{k-1})) + g(v_k)-g(v_{k-1}) + \bigo{\tau_{k-1}^2}.
\end{eqnarray*}
This is greater than zero provided $Cv_{k-1}+g(v_{k-1}>0.$
Therefore $v_{k+1}-v_k>0$ for both $\tau_k$ and $\tau_{k-1}.$ The proof is completed by mathematical induction.
\end{proof}

The above theorem is applicable for a large swath of numerical approximations to reaction diffusion equations of the form (\ref{Introduction:GeneralEq1})-(\ref{Introduction:GeneralEq2}) as long as the source term is positive and monotonically increasing. Hence, a singular nature to the source term is not a necessary requirement and hence the applicability to blow-up problems is obvious.  As such, this theorem suggests that upon formulating the semidiscretized equations, in order to guarantee monotonically increasing numerical solution one merely needs to investigate the positivity of the underlying Pad\'{e} approximate and ensure that $C v(t_0) + g(v(t_0))>0$.

Of course this latter condition is intimately connected to the initial the size of the source function as compared to the matrix $C.$ At first blush it may seem that any arbitrary initial condition $v_{\ell}$ will satisfy $C v_{\ell} + g(v_{\ell})>0.$  Unfortunately, since the nonzero entries of $C$ scale on the square inverse of the minimum spatial sizes then, indeed, $C v_{\ell}$ may be quite large.  Then again, if $v_{\ell}=0$ then this necessary condition is initially satisfied and the above theorem guarantees that the numerical solution will monotonically increase.  Naturally, this condition was anticipated on further inspection of (\ref{GovEqs:MOL}), for which guarantees that $v'(t)>0$ initially.  In either case, for a nonzero initial condition care needs to be taken to ensure this unavoidable condition is satisfied, upon doing so the increasing sequence is guaranteed provided the temporal step is sufficiently small.

The attention is now turned toward the particular matrices involved in the computation in order to verify that the positivity requirement of Theorem 4.1 is satisfied.

\begin{theorem}\label{PnegRneginvpos}
The matrices
\[ (I - \frac{\tau_k}{2}P), ~~(I - \frac{\tau_k}{2}R), \]
are invertible and inverse positive.
\end{theorem}
\begin{proof}
It is simple to verify that both matrices are strictly diagonally dominant. Hence invertibility is guaranteed.  In addition, the main diagonal is positive while all off-diagonal entries are less than or equal to zero.  Therefore both matrices are $M-$matrices and are indeed inverse positive \cite{Poole_1974}.
\end{proof}

Remarkably, no restriction on the temporal or spatial sizes are needed to guaranteed the above result.  The positivity of remaining matrices, however, do restrict the discretization sizes.

\begin{theorem}\label{PposRposthm}
If
\[ \tau_k < \min_{i,j} \{\Delta \mu_i \Delta \mu_{i-1}, \Delta \theta_i \Delta \theta_{i-1} \} \psi_{ii}^{-1}. \]
then matrices
\[ (I+\frac{\tau_k}{2}P), ~~ (I+\frac{\tau_k}{2}R),\]
are nonnegative. Furthermore, if $\tau_k < 1/4 \min_{i,j} \{\Delta \mu_i \Delta \mu_{i-1}, \Delta \theta_i \Delta \theta_{i-1} \} \psi_{ii}^{-1}$ then $(I+\tau_k C)$ is nonnegative.
\end{theorem}
\begin{proof}
Consider the matrix $(I+\frac{\tau_k}{2}P).$  The matrices $B$ and $C_{\mu}$ are positive and the off diagonals of $T_{\mu}>0.$ Consequently, only the main diagonal is of any concern.  Consider the main diagonal, that is,
\begin{eqnarray*}
(I+\frac{\tau_k}{2}P)_{ii} &=& 1 - \frac{\tau_k}{\Delta \mu_i \Delta \mu_{i-1}} \psi_{ii} + \delta_{1i} \frac{\tau_k}{2}\kappa_3 \psi_{1i}, \\
&>&  1 - \frac{\tau_k}{\Delta \mu_i \Delta \mu_{i-1}} \psi_{ii}.
\end{eqnarray*}
Given the assumption on $\tau_k$ this is clearly greater than zero.

Similarly, only the main diagonal entries of $(I+\frac{\tau_k}{2}R)$ are of concern.  Again, with the assumption on $\tau_k$ the result is guaranteed.  A comparable argument holds for the matrix $I+\tau_k C.$
\end{proof}

The above results combined with the satisfied criterion on the temporal size the positivity of the involved matrices is guaranteed.  Therefore by Theorem \ref{GeneralMonotonicThm} the numerical solution will monotonically increase toward unity or till steady state is achieved.

The attention is now turned toward linear stability of the numerical algorithm, that is, showing that any perturbation to the numerical solution is bounded. This amounts to freezing the nonlinear term and investigating the amplification matrix.  Although this is a current limitation to the numerical analysis it has seen reasonable stability results in numerical experiments, in particular in \cite{Beauregard_2012MIT} and the weakly nonlinear analysis presented in \cite{Liang_2006}. In the forthcoming numerical experiments an investigation into the nonlinear stability is offered.  The vector $\infty-$norm us used throughout unless indicated otherwise, defined as,
$$ \| v \|_{\infty} = \max_i |v_i|,$$
for a vector $v$.  For matrices, the matrix $\infty-$norm is also used, that is, for a $n\times n$ matrix $A$,
$$ \| A \|_{\infty} = \max_{i=1,\ldots,n} \sum_{j=1}^n |a_{ij}|.$$

The following lemmas are a staple to the final theorem.

\begin{lemma}
If
\[ \tau_k < \min_{i,j} \{\Delta \mu_i \Delta \mu_{i-1}, \Delta \theta_i \Delta \theta_{i-1} \} \psi_{ii}^{-1} \]
then
\[ \| (I+\frac{\tau_k}{2}P)(I+\frac{\tau_k}{2}R) \| = 1.\]
\end{lemma}
\begin{proof}
Define
\[ \beta_i = 1 + \frac{\tau_k}{2}\sum_{j=1}^{N(M+1)} P_{ij} \mbox{ and } \alpha_i = 1 + \frac{\tau_k}{2}\sum_{j=1}^{N(M+1)}R_{ij},\]
that is, the row sums of the involved matrices in the product. Of course, by Theorem \ref{PposRposthm} both matrices $I+\frac{\tau_k}{2}P$ and $I+\frac{\tau_k}{2}R$ are nonnegative and, therefore, $\beta_i$ and $\alpha_i>0.$  Additionally,
\[ \| (I+\frac{\tau_k}{2}P) \| = \max_i \beta_i \mbox{  and  } \|(I+\frac{\tau_k}{2}R) \| =\max_i \alpha_i .\]
Obviously, as a result of the periodic boundary conditions $\max_i \alpha_i = 1.$ Similarly, $\beta_i=1$ for $i\neq kN$ where $k=1,\ldots, M+1.$ In other words, it is one on all rows outside of where the Dirichlet boundary conditions are implemented.  Now,
\begin{eqnarray*}
0 < \beta_{kN} = 1 -  \frac{\tau_k}{2}B_{kN}\left(\frac{1}{\Delta \mu_{N-1}(\Delta \mu_{N-1}+\Delta \mu_{N})}\right) < 1.
\end{eqnarray*}
Therefore $\max_i \beta_i = 1.$  Consequently,
\[ \| (I+\frac{\tau_k}{2}P)(I+\frac{\tau_k}{2}R) \| \leq 1.\]
Equality is obtained by showing that a particular row sums to one.  For $i\neq kN,$
\begin{eqnarray*}
\sum_{j=1}^{N(M+1)} \left((I+\frac{\tau_k}{2}P)(I+\frac{\tau_k}{2}R)\right)_{ij} &=& \sum_{j=1}^{N(M+1)} \sum_{k=1}^{N(M+1)} \left((I+\frac{\tau_k}{2}P)\right)_{ik} \left((I+\frac{\tau_k}{2}R)\right)_{kj} \\
&=& \sum_{k=1}^{N(M+1)} \sum_{j=1}^{N(M+1)} \left((I+\frac{\tau_k}{2}P)\right)_{ik} \left((I+\frac{\tau_k}{2}R)\right)_{kj}\\
&=&\sum_{k=1}^{N(M+1)} \left((I+\frac{\tau_k}{2}P)\right)_{ik} = 1.
\end{eqnarray*}
Thus, the theorem is clear.
\end{proof}

The above lemma, relies on established positivity of the involved matrices as well as the Pad\'{e} approximation used.  Notably, this proof can be used as a cleaner alternative to the results shown in \cite{Beauregard_2011} for rectangular physical domains.

\begin{lemma}
\[ \| (I-\frac{\tau_k}{2}R)^{-1}(I-\frac{\tau_k}{2}R)^{-1} \| \leq 1.\]
\end{lemma}
\begin{proof}
By Lemma \ref{PnegRneginvpos} the inverses exist. Define
\[ \beta_i = 1 -\frac{\tau_k}{2}\sum_{j=1}^{N(M+1)} P_{ij} \mbox{ and } \alpha_i = 1 - \frac{\tau_k}{2}\sum_{j=1}^{N(M+1)}R_{ij}.\]
As with the previous lemma $\beta_i$ and $\alpha_i$ are one for $i\neq kN.$ In fact, $\alpha_i = 1$ as the individual rom sums of $R$ is zero. Consider,
\begin{eqnarray*}
\alpha_{kN} = 1 + \frac{\tau_k}{2}B_{kN} \left(\frac{1}{\Delta \mu_{N-1}(\Delta \mu_{N-1}+\Delta \mu_{N})}\right) > 1.
\end{eqnarray*}
Thus the $\inf_i \alpha_i = 1.$ Utilizing the Varah bound the theorem is clear \cite{Varah_1975}.
\end{proof}

The combined result of the previous two lemmas enables one to bound the amplification matrix $S(\tau_k).$ This is stated concisely in the following theorem.

\begin{theorem}
If
\[ \tau_k < \min_{i,j} \{\Delta \mu_i \Delta \mu_{i-1}, \Delta \theta_i \Delta \theta_{i-1} \} \psi_{ii}^{-1}. \]
then the adaptive splitting scheme (\ref{NumericalScheme:FullyDiscretized}) with \textbf{fixed nonuniform} grids and the nonlinear term frozen is weakly stable in the von Neumann sense under the $\ell_{\infty}-$norm. That is,
\[ \| z_{k+1} \| \leq \|z_m\|, \]
where $z_m = v_k - \tilde{v}_k$ is a perturbation or error at time step $m$ and $z_{k+1}=v_{k+1}-\tilde{v}_{k+1}$ is the perturbation or error arising after taking $k+1-m$ steps.
\end{theorem}

\section{Numerical Experiments}

The numerical experiments explore the validity of the numerical algorithm in addition to providing an experimental analysis of the singular partial differential equation.  First, for a fixed nonlinear source term and the degeneracy set to unity, the numerical solution is calculated from a zero initial condition. A Monte Carlo method is then utilized to statistically examine the stability and effect of random errors of a certain size that are introduced in the late stages of the computation. In addition, a statistical analysis of the effect of random perturbations randomly introduced throughout the computation is given.  While, this is no substitute for a full nonlinear analysis of the error and stability, it provides a novel way of statistically quantifying these important characteristics and presents a new way of investigating the stability. Hence, this methodology may be employed in other nonlinear problems to better understand the effects of small perturbations.

The remaining experiments analyze the partial differential equation. First, critical quenching domains are calculated for various elliptical shapes. That is, upon fixing the ratio between the minor and major axes the area is reduced till quenching is no longer numerically observed.  The critical quenching domain's area is then stated as the smallest observed domain for which quenching persists for a particular ratio $B/A.$ Lastly, the effect of the degeneracy on the quenching location is examined.  In particular, we investigate the effect of defects on the transportation ($s(x,y)$ vanishes at a boundary point) or diffusion ($s(x,y)$ is singular at a boundary point) of heat.

For each computation the nonlinear source function considered is,

\[ f(u) = \frac{1}{1-u}.\]

The numerical scheme is implemented in Matlab\textsuperscript{\textregistered}, for which a high-level C program is then generated through Matlab Coder$\texttrademark .$ The computation is then carried out on a high performance HP C3000BL HPC cluster running CentOS V at Baylor University.  The processor consists of 128 computer nodes, each with 16 GB of RAM and dual quad-core Intel 2.6 GHz processors giving a total of 1024 cores.

~

\noindent \textit{\textbf{Experiment 1.}} Let the major and minor axis be $6$ and $4$, respectively.  Let $s(x,y)=1$, that is, there is no mathematical degeneracy.  As such, it is reasonable to expect the quenching location to occur at the origin.  The numerical solution is computed utilizing a uniform $101 \times 102$ grid. The initial temporal step is $\tau_0 = .9 \times 10^{-4}$, while $u_0(x,y)=0,$ hence the hypothesis of Theorems 5.1 and 5.3 are satisfied.  The numerical solution quenches at $t=0.861609.$ The temporal step reduces monotonically to the minimum step size of $10^{-8}$ as shown in Figure \ref{fig1adapt}.
\begin{figure}[h]
\begin{center}
\includegraphics[scale=.45]{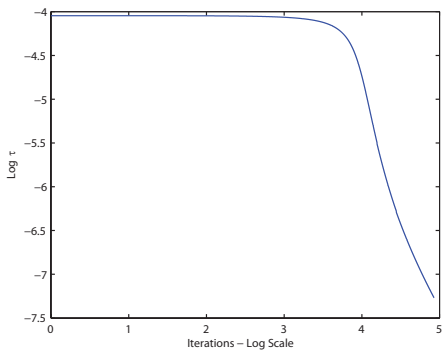}
\caption{A log-log plot of the temporal step throughout the duration of the computation.  The minimum step-size is achieved prior to quenching.  The greatest change in the temporal step occurs later in the computation as expected.}
\label{fig1adapt}
\end{center}
\end{figure}

The solution and the temporal derivative in the iteration just prior to quenching in cartesian and elliptical coordinates are shown in Figures \ref{Example1FigureNoPert}(a)-(d). As expected, x-axis symmetry is evident and the quenching location is found to occur at the origin.

\begin{figure}[t]
\begin{center}
\includegraphics[scale=.45]{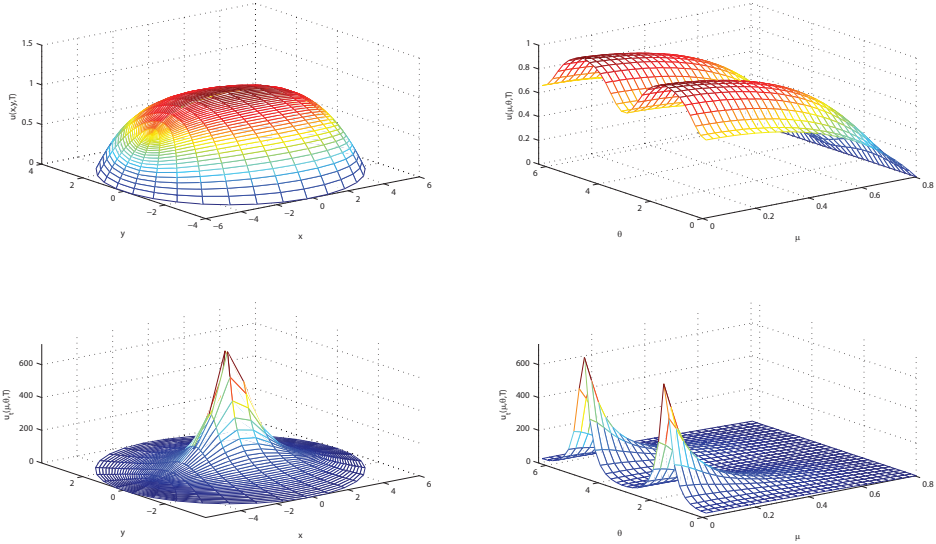}
\caption{(a-b) The numerical solution and its (c-d) temporal derivative in cartesian and elliptical coordinates. The $\max_{x,y}u_t \approx 760$. A third of the grid points are shown to minimize digital memory sizes.}
\label{Example1FigureNoPert}
\end{center}
\end{figure}

The stability analysis of the previous section guaranteed weak stability throughout the computation.  This amounts to freezing the nonlinear term and examining the amplification matrix.  Hence, near the onset of quenching the nonlinear terms dominate and, consequently, there is no such guarantee.  Here, the stability of the numerical scheme throughout the tail end of the computation is examined experimentally.  To accomplish this, a small and random perturbation of order $10^{-n}$ is introduced to the numerical solution found at the iteration for which the $\max v_k$ is first greater than or equal to $0.9.$  The scheme then proceeds with the perturbed solution, that is,
\[ v_k \rightarrow v_k + 10^{-n} z, \]
where $z$ is an order 1 random vector of appropriate size.  A Monte Carlo simulation is utilized to statistically examine the effect on the quenching time, location, and relative difference in the unperturbed $v_k$ and perturbed $\tilde{v}_k$ numerical solutions as the size of the perturbation is increased.  The results are show in the first six rows of Table \ref{Example1Table1}.

\begin{table}[h]
\begin{center}
\begin{tabular}{r|c|c|l|l}\hline
$n$ & (x,~y) & $\tilde{T}$ & $\| v_f-\tilde{v}_f \|$ & $| T-\tilde{T}| / |T|$ \\ \hline
15  & (0.000,~0.000) & 0.861609 & $9.7 \times 10^{-14}$ & 0.0 \\ 
10  & (0.000,~0.000) & 0.861609 & $3.1 \times 10^{-9 }$ & 0.0 \\ 
5   & (0.000,~0.000) & 0.861609 & $6.3 \times 10^{-5 }$ & $1.8 \times 10^{-7}$ \\
4   & (0.000,~0.000) & 0.861608 & $7.8 \times 10^{-4 }$ & $8.4 \times 10^{-7}$ \\
3   & (0.000,~0.000) & 0.861563 & $1.3 \times 10^{-3 }$ & $5.3 \times 10^{-5}$ \\
2   & (0.064,~0.000) & 0.860117 & $6.3 \times 10^{-3 }$ & $1.7 \times 10^{-3}$ \\
NA  & (0.000,~0.000) & 0.861609 & 0.0 & 0.0 \\ \hline \hline
10  & (0.000,~0.000) & 0.861609 & $2.5 \times 10^{-6 }$ & $1.2 \times 10^{-16}$ \\
15  & (0.000,~0.000) & 0.861609 & $6.9 \times 10^{-12}$ & $1.3 \times 10^{-16}$ \\ \hline
\end{tabular}
\caption{The quenching locations $(x,y)$, times $\tilde{T}$, and maximum difference between the perturbed and unperturbed solutions at the quenching time of the perturbed solution.  The relative differences between the quenching times is also given. The last two rows involve random perturbations of order $10^{-10}$ and $10^{-15}$ randomly throughout the entire computation and compares the unperturbed solution.}
\label{Example1Table1}
\end{center}
\end{table}

The scheme is quite robust and the overall effect of the perturbation is roughly of the same order as the size of the initial perturbation itself.  There is a subtle difference in the absolute difference in the computed solutions, but interestingly, there is less influence on the calculated quenching location and time.  In fact, the quenching location shows no observable difference until the order of the perturbation is $10^{-2}.$ Comparable results can be established if the perturbations are introduced earlier in the computation indicating statistical and experimental evidence for stability.

Naturally, one-time perturbations seems contrived and a best-case scenario. It is reasonable to expect errors in the floating point arithmetic that may occur throughout the entire computation.  Let's consider two simulations where perturbations of order $n=10$ and $15$ are introduced randomly throughout the computation. Again, a Monte Carlo technique is employed to statistically quantify the error effects.  The results are shown in the last two rows of \ref{Example1Table1}. Interestingly, for perturbations of order $n=15$ there is excellent agreement with the unperturbed solution, although the effect of the perturbations are now of order $10^{-12}.$  Hence, it seems reasonably to expect that simulations on a machine with 16-digit floating point precision can be established with confidence.  On the other hand, larger perturbations of order $10^{-10}$ may accumulate more and creep into the desired resolution of the computation itself.  Regardless of the size of the perturbations used the quenching location is not influenced and the quenching times see virtually no difference. Relative difference in the perturbed and unperturbed solution stays of the same order for over half of the total iterations, that is, when the nonlinearity becomes of order one.  These statistical results provide additional evidence that the method is indeed reliable and stable to perturbations.

~

\noindent \textbf{\textit{Experiment 2}}  An interesting feature of singular reaction diffusion equations of the quenching type is that for piecewise smooth domains there exists a critical quenching domain \cite{Chan_2011}.  It remains a challenging theoretical problem to go beyond existence theorems and determine bounds on the critical size of a given domain shape.  Of course the results of this paper show that this can be accomplished using knowledge from the rectangular domain case and an established combination of theoretical and numerical results are available in such cases \cite{Beauregard_2011,Chan_1989,Sheng_2003}.  Here, the critical quenching domain is calculated for various ellipses.  To accomplish this, the ratio of minor to major axes ($B/A$) is fixed. Then the area, $\pi AB$ is systematically lowered till quenching is no longer observed.  The critical quenching domain is then established as the smallest area for which numerical quenching is observed.

Without loss of generality the ellipses are centered at the origin and the major axis, $2 A$, is along the x-axis.  No degeneracy is considered, that is, $s(x,y)=1.$ Hence, as in the previous experiment the quenching location is at the origin. The calculations begin from the upper bounds offered in Table \ref{TableQuenchingBoundsTheory}.

\begin{figure}[h]
\begin{center}
\includegraphics[scale=.45]{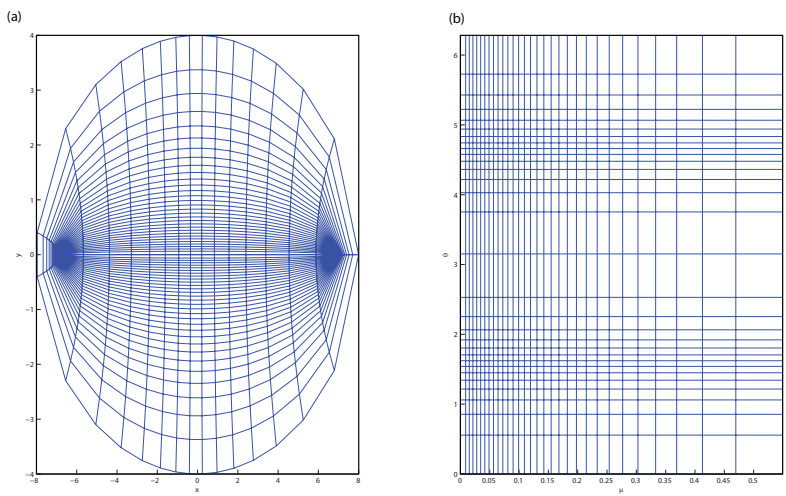}
\caption{An example of an exponentially fitted nonuniform grid in (a) cartesian and (b) elliptical coordinates developed from Experiment 1's quenching solution. A quarter of the grid points are shown to better see the evolution of the grids.}
\label{Example2FigureMesh}
\end{center}
\end{figure}

The fixed nonuniform grids are established utilizing exponential grids \cite{Beauregard_2012MIT} developed from an initial experimentation's quenching solution's temporal derivative. A one sided exponential function is fitted in the least squares sense to the quenching solution's temporal derivative along the line $\theta=\pi/2.$ That function's arc-length is then equidistributed to establish the $\mu$ grid. A two-sided exponential function is utilized in similar fashion along $\mu=0$ for $\theta\in[0,\pi].$ The equidistribution of that function establishes the $\theta$ grid in the first two quadrants.  The remaining grid locations for $\theta$ are chosen such that $\theta_j = \theta_{M+1-j}.$  This results in an exponentially fitted grid offering the finest resolution at the origin; an example is shown in Figure \ref{Example2FigureMesh}.

\begin{figure}[h]
\begin{center}
\includegraphics[scale=.45]{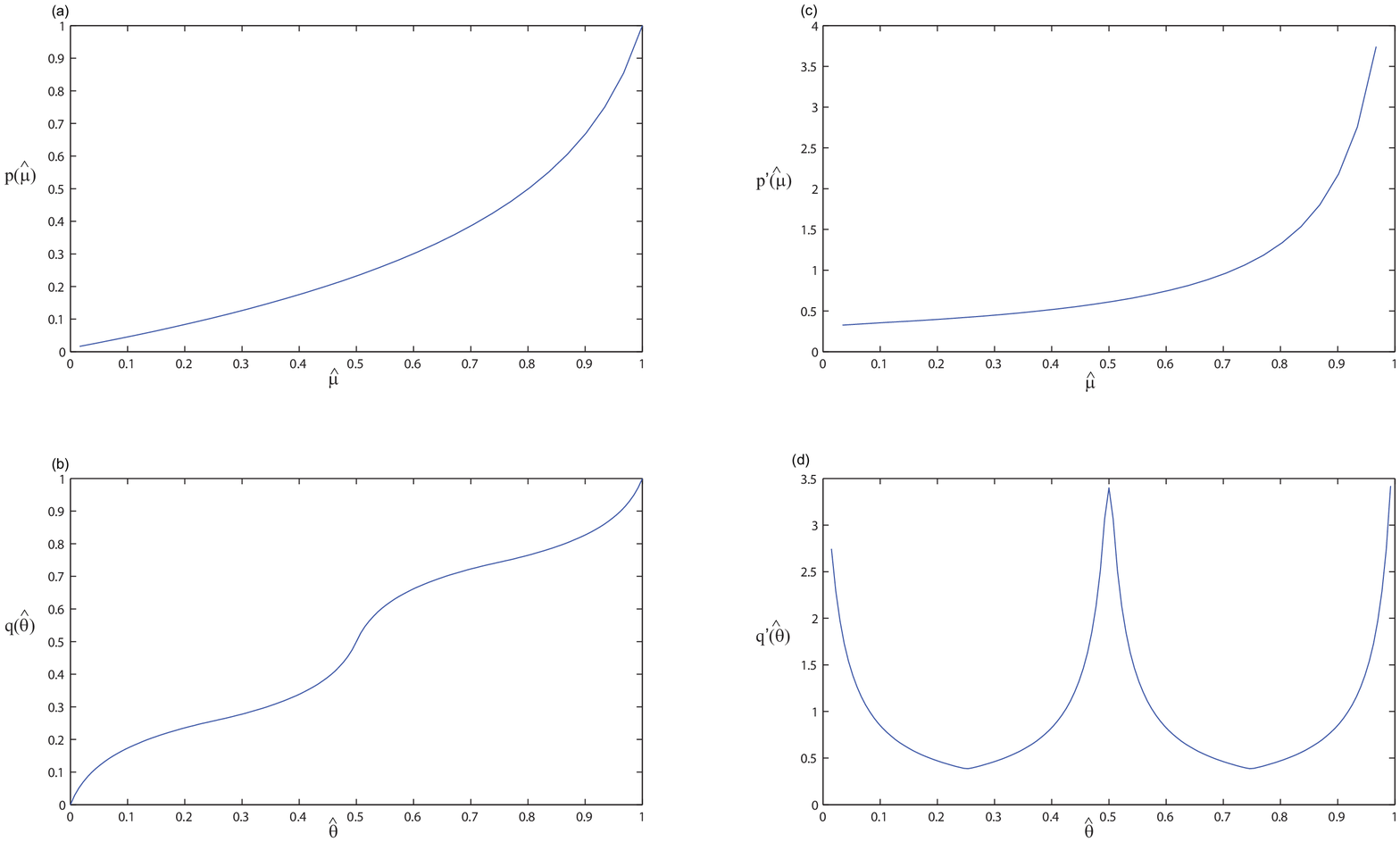}
\caption{The mapping functions (a-b) $p(\hat{\mu})$ and $q(\hat{\theta})$ and their (c-d) derivatives are shown for the mesh shown in Figure \ref{Example2FigureMesh}.}
\label{Example2FigureMeshSmooth}
\end{center}
\end{figure}

Let $p: \hat{\mu} \mapsto \mu$ be a mapping function from the logical space (uniform grid) to the physical space (nonuniform grid) in $\mu$.  Similarly, let $q: \hat{\theta} \mapsto \theta$ be corresponding mapping function for the theta domain.  It is known that the truncation error of a nonuniform central difference method is of first order, however, if the mapping function is smooth then the actual truncation error is second order \cite{Budd_2009}.  The mapping functions are examined for each formulated mesh and are observed to be smooth.  An illustration of the mapping functions and their derivatives for the meshes shown in Figures \ref{Example2FigureMeshSmooth}(a-d).

\begin{table}[h]
\begin{center}
\begin{tabular}{r|r|r|r|r}\hline
$B/A$ & Rect. Area & $T_{Rect}$ & Ellip. Area & $T_{Ellip}$  \\ \hline
.125  & 18.8054    & 44.912 & 25.2282 &  13.030 \\ \hline
.250  &  9.6722    & 47.600 & 13.2296 &  45.367 \\ \hline
.375  &  6.8501    & 48.847 &  9.6480 &  71.578 \\ \hline
.500  &  5.5986    & 49.070 &  8.0503 & 142.985 \\ \hline
.625  &  4.9679    & 49.983 &  7.2480 & 315.403 \\ \hline
.750  &  4.6453    & 32.151 &  6.8562 &  95.962 \\ \hline
.875  &  4.4964    & 29.281 &  6.7133 & 162.050 \\ \hline
\end{tabular}
\caption{The critical quenching domain area $\pi BA$ and time for ellipses of fixed ratios of minor to major axes.  The corresponding quenching domain size and time for rectangles with fixed ratio to width to length is also offered as a comparison.}
\label{Example2Table1}
\end{center}
\end{table}

~

\noindent \textbf{\textit{Experiment 3.}} The mathematical degeneracy encompasses interesting physical phenomena in addition to an intellectual interesting nuance to such singular partial differential equations.  In the case of rectangular domains with Dirichlet boundary conditions, it has been observed that the quenching location shifts toward the degeneracy, that is, the location for which $s(x,y)=0$ on the boundary \cite{Beauregard_2012MIT}.  Experimentally, it is observed that the quenching location is shifted.  Consequently, a mathematical proof has yet to be established and is recommended as a novel undertaking.

In this light, the effect of a mathematical degeneracy is examined here for two situations.  If $s(x,y)=0$ at a boundary point, then this amounts to a degeneracy in the transportation of heat, as the coefficient of $u_t$ vanishes.  In this case, as a result of Theorem 5.3 and Lemma 5.4, the maximum temporal step is more restricted.  In contrast, if $s(x,y)\rightarrow \infty$ at a boundary point, then this amount to a degeneracy in the diffusion of heat, as the coefficient of the Laplacian vanishes. In such case, the maximal temporal step is larger than without a degeneracy.  Both situations are considered here.  Let
\[ q(x,y)= \frac{\gamma}{2}\left( \frac{x^*}{a^2 \cosh^2(\zeta)}(x^*-x) + \frac{y^*}{a^2 \sinh^2(\zeta)}(y^*-y) \right),\]
where
\begin{eqnarray*}
x^*&=&a \cosh(\zeta)\cos(\theta^*), \\
y^*&=&a \sinh(\zeta)\sin(\theta^*), \\
\end{eqnarray*}
$\gamma>0$ is the maximum value of $q(x,y)$ restricted to $\Omega$, and $(x^*,y^*)$ is chosen on $\partial \Omega.$ The projection of plane's normal onto the $z=0$ plane is chosen such that it is perpendicular to the tangent line at the point of degeneracy.  A major and minor axes of $A=8$ and $B=6$, respectively, is used throughout the discussion. Other ratios were explored and share similar observations.

Let $s(x,y)=q(x,y)$ and $\gamma=1.$ In such case, the defect contributes to the overall size of the source function in the numerical computation.  Hence, the quenching time is anticipated to be lower compared to the degenerate-free case.  In fact, for the particular domain shape and size the quenching times for the degenerate and degenerate free cases are determined to be $0.860857$ and $0.278238,$ respectively, in line with this conjecture.

It has been observed numerically that for rectangular domains the quenching location shifts toward the degeneracy $(x^*,y^*).$ Consider the cases $\theta^* = n \pi/4$ for $n=0,\ldots, 7.$  The computation utilizes fine fixed uniform grids to determine the quenching location.  Remarkably, for a degeneracy located at $\theta^*$ on the boundary, the quenching location is the same distance $d$ in the inward normal direction for each computation! In other words, the quenching location for each case of $\theta^*$ forms an inner annulus, which is not necessarily an ellipse. More precisely, the quenching location $(\tilde{x},\tilde{y})$ is
 \begin{eqnarray*}
\tilde{x} &=& \frac{1-d B^2}{\sqrt{B^4 (x^*)^2 + A^4 (y^*)^2}}x^*, \\
\tilde{y} &=& \frac{1-d A^2}{\sqrt{B^4 (x^*)^2 + A^4 (y^*)^2}}y^*,
\end{eqnarray*}
where $(x^*,y^*)$ is the degeneracy location and $d$ is a fixed constant.  Clearly this observation depends on the type of degeneracy, however it indicates that the curvature of the boundary at the point of degeneracy plays a delicate role in influencing the quenching location.  Figure \ref{Example3Loc}(a) shows the elliptical boundary and the corresponding quenching locations for each value of $n.$  The dashed line is the corresponding ellipse with the same shape as the boundary and passing through the quenching locations for $n=0$ and $4.$ Figure \ref{Example3Sol}(a) shows a contour plot of the solution in the iteration just prior to quenching in the $n=1$ case.

\begin{figure}[ht!]
\begin{center}
\includegraphics[scale=.35]{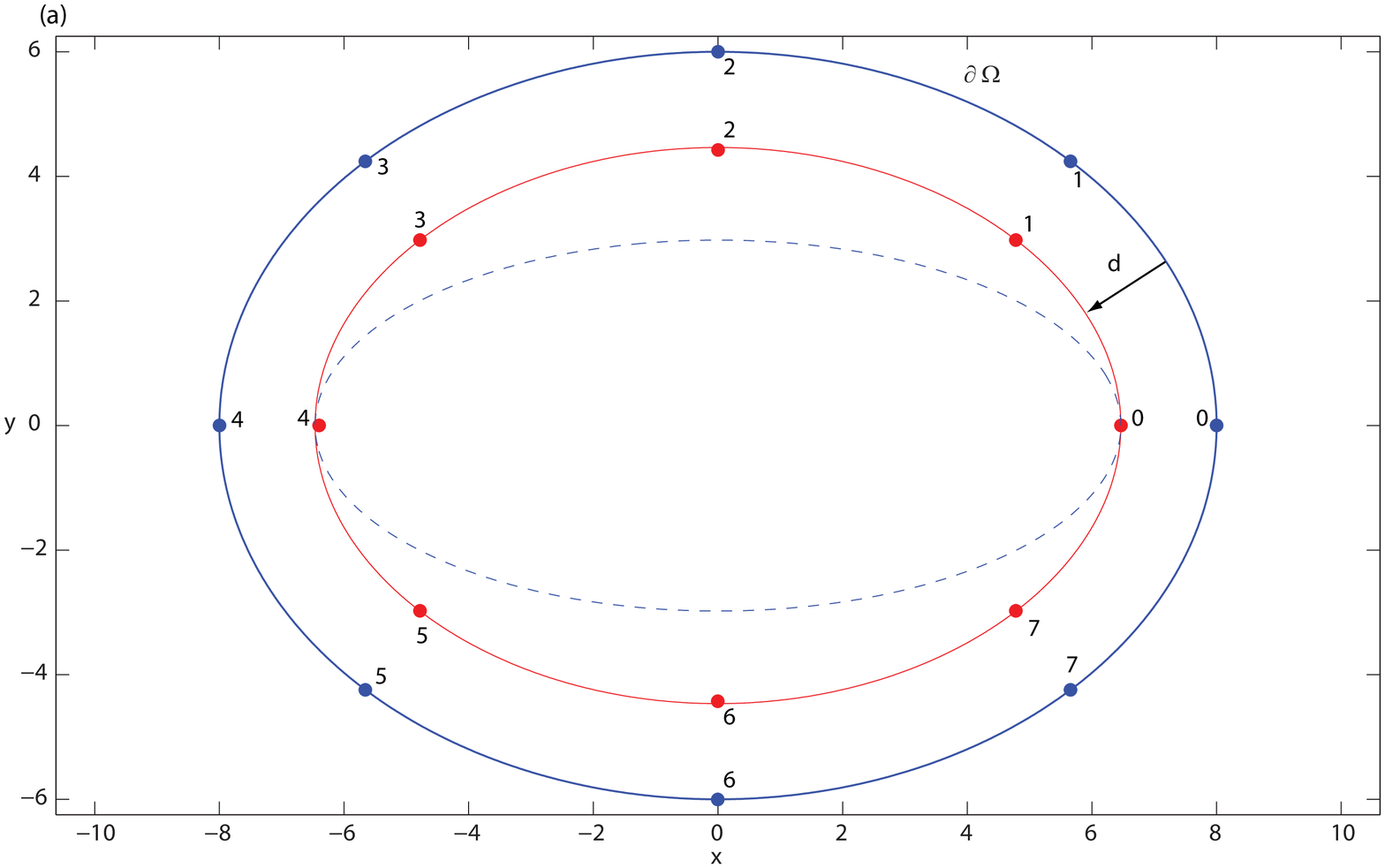}
\includegraphics[scale=.35]{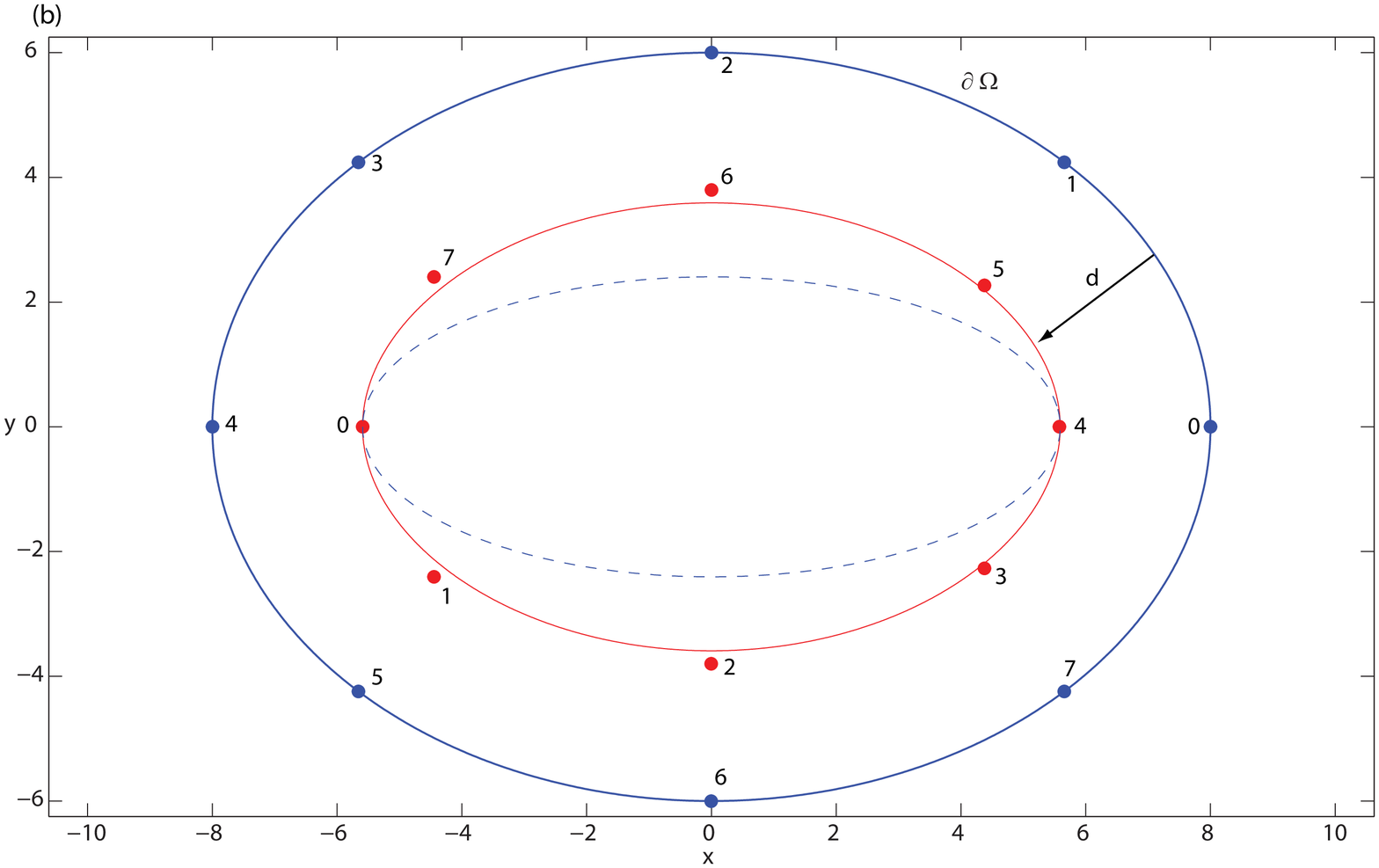}
\caption{The quenching locations for various defects in the (a) transportation and (b) diffusion of heat.  The outer boundary ($\partial \Omega$) is shown. A dashed inner elliptical boundary that passes through the $n=0$ and $n=4$ quenching locations is given for scale.  The annulus given by a fixed distance $d$ from the outer boundary is given, for which the quenching locations closely follow.  The number located at each enclosed circle pairs the corresponding degeneracy with the quenching location for $n=0,\ldots,7.$}
\label{Example3Loc}
\end{center}
\end{figure}

Let $s(x,y)=1/q(x,y).$ Again, the single defect is located on the boundary in increments of $\pi/4.$ At these locations $s(x,y)$ is singular, effectively lessening the diffusive effects. Hence, it is still expected that the quenching time will be less than the degenerate-free case.  However, it will not be more than the previous case of the defect in the transportation of heat. This is because the size of the source term is lessened.  For this particular shape, the quenching time is calculated to be $0.571324.$

Interestingly, a similar pattern emerges in this case.  Again, the quenching locations closely fall on an inner annulus constructed by tracing a path a distance $d$ from the outer boundary.  However, the quenching location shifts away from the degeneracy rather than toward it.  These results are shown in Figure \ref{Example3Loc}(b).  A contour plot of the solution just prior to quenching is also shown in Figure \ref{Example3Sol}(b).

\begin{figure}[h]
\begin{center}
\includegraphics[scale=.35]{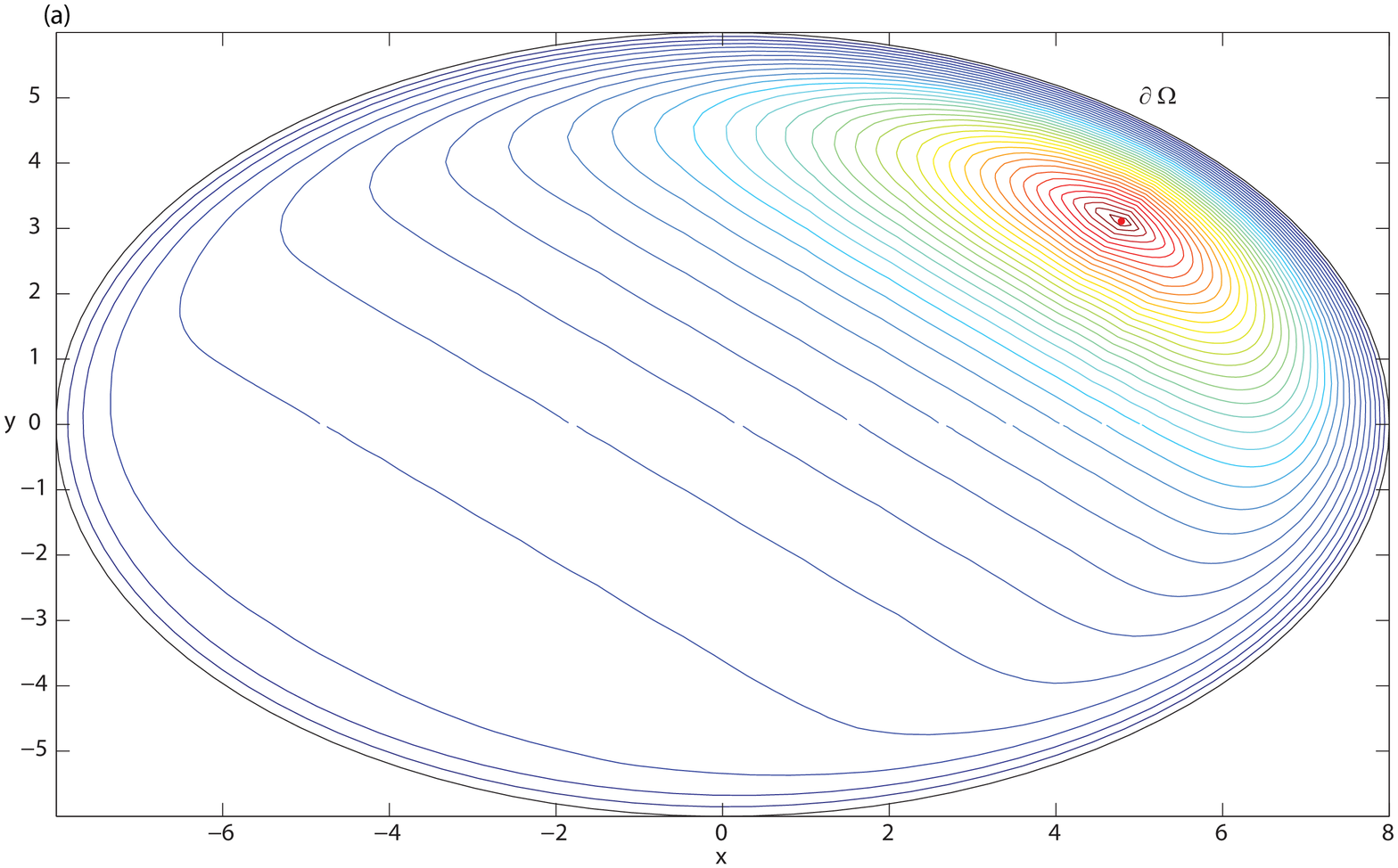}
\includegraphics[scale=.35]{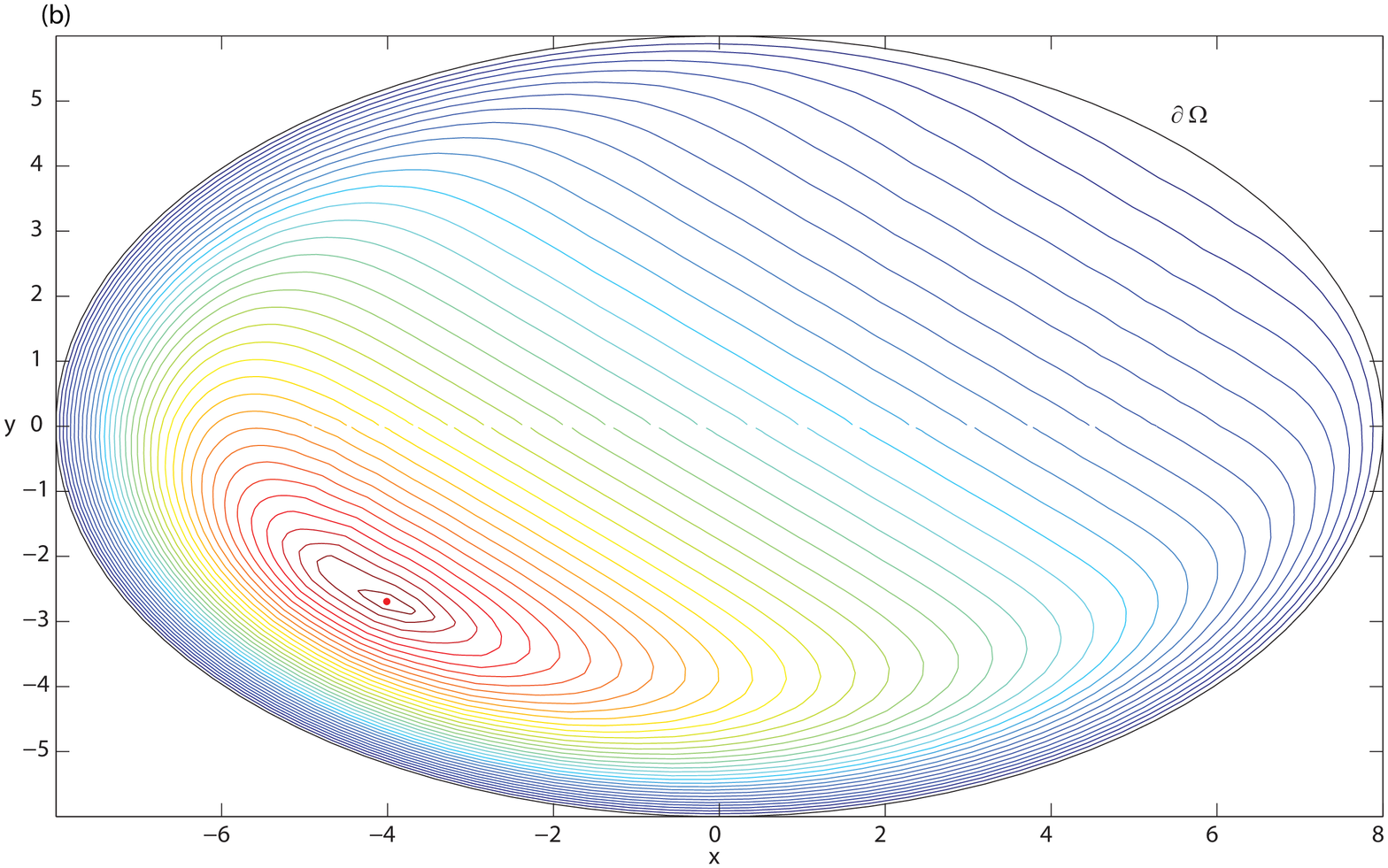}
\caption{Contour plots of numerical solution for defects in the (a) transportation and (b) diffusion of heat.  In each case the defect is located on the boundary at $\theta^*=\pi/4.$ The quenching location is also identified by a dot in each of the plots.}
\label{Example3Sol}
\end{center}
\end{figure}

\section{Conclusions}

This paper designs, implements, and analyzes a temporally adaptive splitting scheme that utilizes nonuniform grids for singular reaction diffusion equations with a nonlinear absorption term.  The multi-faceted quenching phenomena presents challenges to resolving dynamics close to the onset of quenching, in determining the effect of defects on the quenching location and dynamics of the solution, and in determining critical quenching sizes for particular domain shapes.  This paper presents improvements to our understanding in all of these categories adding interesting general results while, at the heart of the paper, in application to elliptical domains.

Determining theoretical estimates to the critical quenching domain size for a given shape is a delicate and difficult task.  This is evidenced by the literature devotion to existence theorems of critical quenching domain, rather than estimates of the size.  Only for rectangular domains such estimates have surfaced theoretically and numerically. This focus is not wasted, rather is a blessing to exploring other domain shapes.  In particular, the acquired knowledge of critical quenching domains for rectangular domains can be used to construct lower and upper solutions for quenching problems posed over arbitrary piecewise smooth domains.  This is then used to construct theoretical estimates to the critical quenching domain sizes. In this regard, this result can be extended to blow-up problems as well.  Additionally, through comparable techniques it is our belief that sharp estimates can be established and is a target of our future research.

The governing equations are posed over an elliptical domain and the numerical scheme is subsequently developed.  First, a coordinate transformation maps the elliptical domain into a rectangular one.  This requires an additional boundary condition, which is realized through investigating certain translational symmetries in the coordinate transformation.  The proposed boundary condition can be implemented in a straight-forward way in the spatial discretization.  The diffusion operators are discretized using nonuniform central differences.  The current study considers fixed grids, however the grids are developed through an adaptive procedure called exponentially evolving grids. The numerical solution at any iterate provides information in order to develop a least squares approximation to the solution's temporal derivative.  The collocation function is chosen such that certain properties are available, namely the finest resolution is placed near the location of the solution's extrema.  Experiments suggest this approach develops smooth mapping functions from the logical to the physical spaces, indicating a second order accuracy.  Future studies are under way to better understand this type of adaptation and, hence, is reserved for a later follow-up paper.

The ensuing numerical analysis indicates that provided certain criteria are met, that the numerical solution monotonically increases and remains positive throughout the computation.  A general theorem is provided that can be used to establish this result for any problem that its semi-discretization is of the form of (\ref{GovEqs:MOL}). Hence, while the primary discussion is focused on quenching problems over elliptical domains, the theorem makes only use of the form of the semi-discretization and provides criteria for the involved matrices.  Weak stability is then established in the von Neumann sense in the $\infty-$norm. While this lacks the full nonlinear stability analysis the experiments provide a novel way of quantifying the effects of error throughout a computation in a statistical sense.  It is true that for nonlinear problems, the stability analysis of any developed scheme is a difficult task.  Often, weak stability is all that is offered.  However, the cumulative effect of error can be statistically analyzed using Monte Carlo methods.  It is our belief that this methodology could be used as an alternative that offers valid experimental and statistical evidence toward the stability of the numerical scheme.  In this light, the experiments indicate remarkable stability to continued perturbations throughout the computation.

Defects to the transportation and diffusion of heat are physically motivated in addition to mathematically interesting to their effect on the solution to the partial differential equation.  Here, the scheme is used to analyze the effect of such defects.  It has been observed that in the context of rectangular domains defects in the transportation of heat results in movement of the quenching location toward the defect.  This phenomena is observed for elliptical domains.  To our knowledge, no such study has been presented on the effect of defects in the diffusion of heat.  Interestingly, the quenching location shifts away from the defect! While these results have been established numerically it has yet to be seen theoretically and is left as an open problem.


\begin{thebibliography}{10}

\bibitem{Bandle_1998} C. Bandle and H. Brunner, {\em Blowup in diffusion equations: a survey}, J. Comput. Appl. Math. 97 (1998), pp.~3-22.

\bibitem{Beauregard_2011} M.\,A. Beauregard and Q. Sheng, {\em An Adaptive Splitting Approach for the Quenching Solution of Reaction-Diffusion Equations over Nonuniform Grids}, J. Comput. Appl. Math., 241 (2013), pp.~30-44.

\bibitem{Beauregard_2012MIT} M.\,A. Beauregard and Q. Sheng, {\em Solving degenerate quenching-combustion equations by an adaptive splitting method on
evolving grids},  Comput. Struct., 122 (2013), pp.~33-43.

\bibitem{Beauregard_2013CIRC} M.\,A. Beauregard and Q. Sheng, {\em A fully adaptive method to approximate reaction diffusion equations of the quenching type over circular domains}, Numer. Meth Partial Diff. Eqns, DOI 10.1002/num.21820, 2013.

\bibitem{Bebernes_1989} J. Bebernes and D. Eberly, {\em Mathematical Problems from Combustion Theory}, Springer-
Verlag, Berlin and New York, 1989.

\bibitem{Boor_1973} C. de Boor, {\em Good Approximation by Splines with Variable Knots II\/,} Springer, Berlin, 1974.

\bibitem{Budd_2009} Budd C, Huang W, Russell RD. {\em Adaptivity with moving grids}, Acta Numerica 18 (2009), pp.~111-241.

\bibitem{Chan_1989} C.Y. Chan and C.S. Chen, {\em A numerical method for semilinear singular parabolic mixed boundary-value problems}, Quart. Appl. Math. 47 (1989), pp.~45-57.

\bibitem{Chan_1994} C.\,Y. Chan and L. Ke, {\em Parabolic quenching for nonsmooth convex domains}, J. Math. Anal. Appl. 186 (1994), pp.~52–65.

\bibitem{Chan_2011} C.\,Y. Chan, {\em A quenching criterion for a multi-dimensional parabolic problem due to a concentrated nonlinear source}, J. Comput. Appl. Math., 235 (2011), pp.~3724-3727.

\bibitem{Sheng_2003} H. Cheng, P. Lin, Q. Sheng and R.C.E. Tan, {\em Solving degenerate reaction-diffusion equations via variable step Peaceman-Rachford splitting}, SIAM J. Sci. Comput. 25 (2003), pp.~1273-1292.

\bibitem{Horn_MatrixAnalysis} R.\,A. Horn and C.\,R. Johnson, {\em Matrix Analysis}, Cambridge University Press, New York, 1985.

\bibitem{Kawarada_1975} H. Kawarada, {\em On solutions of initial-boundary problem for $u_t=u_{xx}+\frac{1}{1-u}$}, Publ. Res. Inst. Math. Sci., 10 (1975), pp.~729-736.

\bibitem{Kirk_2002} C. Kirk and W. Olmstead, {\em Blow-up in a reactive-diffusive medium with
a moving heat source}, J. Zeitschrift Angew. Math. Phys., 53 (2002), pp.~147-159.

\bibitem{Levine_1980} H. Levine and J.\,T. Montgomery, {\em The quenching of solutions of some nonlinear parabolic equations}, SIAM J. Math. Anal., 11 (1980), pp.~842-847.

\bibitem{Lia_2004} M. Lai, {\em Fast direct solver for Poisson equation in a 2D elliptical domain}, Numer. Methods Partial Differential Equations. 20 (2004), pp.~72-81.

\bibitem{Liang_2006} K. Liang, P. Lin, M. Ong, and R. Tan, {\em A splitting moving mesh method for reaction-diffusion equations of quenching type}, J. Comp. Phys., 215 (2006), pp.~757-777.

\bibitem{Nouaili_2011} N. Nouaili, {\em A Liouville theorem for a heat equation and aplications for quenching}, Nonlinearity, 24 (2011), pp.~797-832.

\bibitem{Poinsot_2005} T. Poinsot and D. Veynante,
{\em Theoretical and Numerical Combustion}, Edwards Publisher, Philadelphia, 2005.

\bibitem{Poole_1974} G. Poole and T. Boullion, {\em A survey on M-matrices}, SIAM Review, 16 (1974), pp.~419-427.

\bibitem{Qiao_2011} Z. Qiao, Z. Zhang and T. Tang, {\em An adaptive time-stepping strategy for the molecular beam epitaxy models},  SIAM J. Sci. Comput., 33 (2011), pp.~1395-1414.

\bibitem{Sattinger_1973} D. Sattinger, {\em Topics in stability and bifurcation theory,} Lecture Notes Mathematics, Vol. 309, Springer, 1973.

\bibitem{Schatzman_1999} M. Schatzman, {\em Stability of the Peaceman-Rachford approximation}, J. Funct. Anal., 162 (1999),
pp.~219-255.

\bibitem{Sheng_2012} Q. Sheng and A.\,Q.\,M. Khaliq,
{\em A revisit of the semi-adaptive method for singular degenerate
reaction-diffusion equations}, East Asia J. Appl. Math., 2 (2012), pp.~185-203.

\bibitem{Walter_1970} W. Walter, {\em Differential and Integral Inequalities}, Ergebnisse der Mathematik und ihrer Grenzgehiete, Springer-Verlag, 1970.

\bibitem{Varah_1975} J. Varah, {\em A lower bound for the smallest singular value of a matrix}, Lin. Alg. Appl. 11 (1975), pp.~3-5.

\end{thebibliography}
\end{document}